\documentclass[a4paper,leqno]{article}
\usepackage[latin1]{inputenc}
\usepackage{amsfonts}
\usepackage{enumitem}
\usepackage{color}
\usepackage[mathscr]{eucal}
\usepackage{textcomp}
\usepackage [english]{babel}
\usepackage{graphicx,dsfont}
\usepackage{stmaryrd}
\usepackage{amssymb,amsmath,amsthm,esint}

\usepackage[notcite,notref]{showkeys}
\newtheorem{define}{Definition}[section]
\newtheorem{pro}{Proposition}[section]

\newtheorem{cor}{Corollary}[section]

\newtheorem{remark}{Remark}[section]
\theoremstyle{plain} \newtheorem{thm}{Theorem}[section]
\DeclareMathOperator{\IT}{Int\,}

\catcode`\@=11
\makeatletter

\@addtoreset{equation}{section}
\def\R{\mathbb R}

\usepackage[left=3cm,top=2cm,right=2.3cm,nohead]{geometry}
\title{Asymptotic behavior of blowing-up radial solutions for quasilinear elliptic
systems arising in the study of viscous, heat conducting fluids}

\author{Ahmed Bachir, Jacques Giacomoni and Guillaume Warnault}

\date{\today}

\begin{document}
\maketitle
\begin{abstract}
   In this paper, we deal with  the following quasilinear elliptic system involving gradient terms
in the form:
\begin{center}
$\begin{cases}
\Delta_p u= v^m| \nabla u |^\alpha& \text{in}\quad \Omega\\
\Delta_p v= v^\beta| \nabla u |^q & \text{in}\quad \Omega,
\end{cases}$
\end{center}
where $\Omega\subset\mathbb{R}^N(N\geq 2)$ is either equal to $ \mathbb{R}^N $ or equal to  a ball $B_R$ centered at the origin and having radius $R>0$, $1<p<\infty$, $m,q>0$, $\alpha\geq 0$, $0\leq \beta\leq  m$ and $\delta:=(p-1-\alpha)(p-1-\beta)-qm \neq 0$. Our aim is to establish the asymptotics of the blowing-up radial solutions to the above system.  Precisely, we provide the accurate asymptotic behavior at the boundary for such blowing-up radial solutions. For that,we prove a strong maximal principle for the problem of independent interest and study an auxiliary asymptotically autonomous system in $\R^3$.
\end{abstract}
\section{Introduction and main results}		
This paper deals with a class of quasilinear elliptic system of the following type:
\begin{equation}\label{otoo23}
\begin{cases}
\Delta_p u= v^m| \nabla u |^\alpha& \text{in}\quad \Omega\\
\Delta_p v= v^\beta| \nabla u |^q & \text{in}\quad \Omega,
\end{cases}
\end{equation}	
Here the operator $\Delta_p$ represents the standard $p$-Laplace operator, and $\Omega\subset\mathbb{R}^N(N\geq 2)$ is either equal to $ \mathbb{R}^N $ if the solution is global or equal to  a ball $B_R$ if the solution blows up at some $\widehat{R} \geq  R$ with $\widehat{R}<\infty$, and the parameters involved in  $\eqref{otoo23}$ satisfy 
\begin{center}
$1<p<\infty,\quad m,q>0,\quad \alpha\geq 0,\quad 0\leq \beta\leq  m.$
\end{center}
The global homogeneity parameter related to the system $ \eqref{otoo23} $, $ \delta $, verifies additionally
\begin{equation}\label{otoo24}
\delta:=(p-1-\alpha)(p-1-\beta)-mq \neq 0.
\end{equation}
Our goal is to determine the behaviour of radial and positive blowing-up solutions at boundary. In this regard, the present paper can be considered as the subsequent continuation  of $\cite{book17}$, where several results about existence, global behavior and uniqueness are obtained for global solutions to the system \eqref{otoo24} (see also \cite{book36} and \cite{FiVi} for the particular case $p=2$).
The semilinear case $p = 2$, $\alpha= \beta=0$, $m = 1$, $q = 2$ of problem $ \eqref{otoo23} $ was considered in details by Diaz, Lazzo and Schmidt $ \cite{book1} $ as a model arising from the study of dynamics of a viscous, heat-conducting fluid. Considering a unidirectional flow, independent of distance in the flow direction under the Boussinesq approximation, the velocity $u$ and the temperature $\theta$ satisfy the following system
\begin{equation}\label{otoo140}
\begin{cases}
u_t -\Delta u= \theta & \text{in}\quad \Omega \times (0,T),\\
\theta_t -\Delta \theta= | \nabla u |^2 & \text{in}\quad \Omega \times (0,T).
\end{cases}
\end{equation}
The source terms $ \theta $ and $ | \nabla u |^2 $ represent the buoyancy force and viscous heating, respectively.
With the change of variable $v = -\theta$, steady states of $ \eqref{otoo140} $ satisfy the following coupled elliptic equations:
\begin{equation}\label{otoo140b}
\begin{cases}
\Delta u= v & \text{in}\quad \Omega \times (0,T),\\
\Delta v= | \nabla u |^2 & \text{in}\quad \Omega \times (0,T),
\end{cases}
\end{equation}
which corresponds to the semilinear case $p = 2$, $\alpha= \beta=0$, $m = 1$, $q = 2$ of Problem $ \eqref{otoo23} $.\\
%We refer \cite
%Let us consider the equation
%\begin{equation}\label{otoo141}
%\Delta_p u= u^m| \nabla u |^\alpha \quad \text{in}\quad \mathbb{R}^N, N\geq  2.
%\end{equation}
Elliptic and parabolic equations with gradient absorption term arise in different models : reaction diffusion equations (see \cite{book24}) in non newtonian fluids models, turbulent flows in porous media (see \cite{book26}). We refer to \cite{book17} for further explanations of physical models.
In \cite{book17} the authors studied nonconstant positive radial solutions of \eqref{otoo23}, that is, solutions $(u,v)$ which satisfy
\begin{itemize}
\item $u,\, v \in C^1(\overline{\Omega})\cap C^2(\overline{\Omega} \setminus \lbrace 0 \rbrace)$ are positive and radially symmetric;
\item $u$ and $ v $ are not constant in any neighbourhood of the origin;
\item $ u $ and $ v $ satisfy \eqref{otoo23}.
\end{itemize}
If $\Omega = \mathbb{R^N}$, solutions of \eqref{otoo23} will be called global solutions.\\ 
Theorem \ref{khaled4} given below is one of main results in  \cite{book17}, that provides a Keller-Osserman type condition (see \cite{Ke} and \cite{Os}) for system \eqref{otoo23} and then classifies all nonconstant positive radial solutions in a ball $B_R$ according to their asymptotic behavior at the boundary.
\begin{thm}\label{khaled4}
Assume that $\Omega = B_R$, $ 1 < p < \infty$, $m, q > 0$, $0 \leq  \alpha< p-1$, $ 0 \leq  \beta \leq  m $ and $\delta\neq 0$. Then,
\begin{enumerate}
\item[(i)]There are no positive radial solutions $(u,v)$with $  u(R^-) = \infty $ and $ v(R^-) < \infty$.
\item[(ii)] All positive radial solutions of $ \eqref{otoo23} $ are bounded {\it i.i.f.} $\delta>0$
%$$mq < (p-1-\alpha)(p-1-\beta).$$
\item[(iii)] There are positive radial solutions $(u,v)$ of $\eqref{otoo23}$ with $u(R^-) < \infty$ and $ v(R^-) = \infty$ {\it i.i.f.} $\delta<-mp-(p-1-\beta)$.
%$$mq > mp + (p-\alpha)(p-1-\beta).$$
\item[(iv)] There are positive radial solutions $(u,v)$ of $\eqref{otoo23}$ with $u(R^-) = v(R^-) = \infty$ {\it i.i.f.}
$\delta \in [-mp-(p-1-\beta),0)$.
%$$(p-1-\alpha)(p-1-\beta) < mq \leq  mp + (p-\alpha)(p-1-\beta).$$
\end{enumerate}
\end{thm}
From the seminal works of \cite{Bi}, Keller-Osserman condition was investigated for a very large class of elliptic equations, systems and inequations. We refer to \cite{GhRa} and   \cite{Ra} for a survey of the corresponding results.The second result in $ \cite{book17} $ concerns the existence of nonconstant global positive radial solutions of $\eqref{otoo23}$, which shows that Theorem $ \eqref{khaled4} $-(ii) is sharp regarding the existence of global solutions:
\begin{thm}\label{khaled5}
Assume that $ \Omega = \mathbb{R}^N$, $p > 1$, $m,\, q > 0$,  $\alpha \geq  0$, $0\leq \beta \leq  m$ and $ \delta \neq 0$. Then, $ \eqref{otoo23} $ admits nonconstant global positive radial solutions if and only if
\begin{equation}\label{otoo31}
0 \leq  \alpha < p-1 \hspace{0.4 cm} and \hspace{0.4 cm} \delta>0.
\end{equation}
\end{thm}
In $ \cite{book17} $, the asymptotic behavior and the uniqueness of global solutions are also investigated.
As it  is observed in \cite{book17}, by taking advantage of the homogeneity of system \eqref{otoo23}, we have that any nonconstant positive radial solution $(u,v)$ of $ \eqref{otoo23} $ in a ball $B_{R}$ satisfies the following system 
\begin{equation}\label{otoo35}
\begin{cases}
\left( (u')^{p-1} \right)' +\dfrac{N-1}{r}(u')^{p-1}=v^m(u')^{\alpha} & \text{in} \quad (0,R), \\
\vspace{0.2cm}
\left( (v')^{p-1} \right)' +\dfrac{N-1}{r}(v')^{p-1}=v^\beta(u')^{q} & \text{in} \quad (0,R),\\
u'(0)=v'(0)=0,\ u,\,v>0& \text{in} \quad (0,R).
\end{cases}
\end{equation}
and $u'$ and $v'$ are increasing on $(0,R)$.\\
In the present paper, to complement the above results, we focus on the asymptotic behaviour of blowing-up solutions. Thanks to the homogeneity property of the system, \eqref{otoo35} can be reduced to the problem
\begin{equation}\label{otoo36}
\begin{cases}
\left( (u')^{p-1-\alpha}\right)' +\dfrac{\gamma}{r}(u')^{p-1-\alpha}=\dfrac{\gamma}{N-1}v^m
& \text{in} \quad (0,R), \\
\vspace{0.2cm}
\left( (v')^{p-1}\right)' +\dfrac{N-1}{r}(v')^{p-1}=v^\beta(u')^{q} & \text{in} \quad (0,R),\\
u'(0)=v'(0)=0,\ u,\,v>0& \text{in} \quad (0,R),
\end{cases}
\end{equation}
where
\begin{equation}\label{otoo37}
\gamma=\dfrac{(N-1)(p-1-\alpha)}{p-1}.
\end{equation}
Let  $(u,v)$ be a nonconstant positive radially symmetric solution of \eqref{otoo36} in the ball $B_R$, then, the pair 
\begin{equation}\label{scal}
(u_\lambda,v_\lambda)=\left(\lambda^{1-\alpha_0}  u\left( \dfrac{x}{\lambda}\right),\, \lambda^{-\beta_0} v \left( \dfrac{x}{\lambda}\right)\right)
\end{equation} 
where $\alpha_0=\dfrac{1+\beta-p(m+1)}{\delta}$ and $\beta_0=-\dfrac{p(p-1-\alpha)+q}{\delta}$, provides a nonconstant positive radially symmetric solution of \eqref{otoo36} in the ball $B_{\lambda R}$. This guarantees that in any ball of positive radius there is nonconstant positive radially symmetric solution of \eqref{otoo36}. Hence, thanks to this scaling property, $ R $ can be arbitrary, then without loss of generality we fix in the sequel $ R=1 $.\\
\ \\
Considering $ R=1 $, the main result of the present paper yields the asymptotic behavior of the blowing-up solutions of \eqref{otoo23} namely when {\it (iii)} or {\it(iv)} holds.
\begin{thm}\label{theorem1}
Under the conditions of Theorem \ref{khaled5}, assume $\delta<0$ and $\alpha_0>0$.
%$p\in (1,2]$, $ q \geq  p-1-\alpha $, $ \beta \geq  1 $, $ m \geq  1 $ (which imply $\delta<0)$ and $\alpha_0>0$. 
Then, 
$$u'(r) \sim \dfrac{\lambda}{(1-r)^{\alpha_0}} \mbox{  and  }\ v(r) \sim \dfrac{\mu}{(1-r)^{\beta_0}} \mbox{  as  $r \to 1^-$} $$
where $\lambda$ and $\mu$ are positive constants given in \eqref{eq13}.
%$$A=  (p-1)\gamma_0(\alpha_0(p-1))^{\frac{q}{p-1-\alpha}}\beta_{0}^{p-1} \mbox{ and }\, \gamma_{0}=1+\beta_0. $$
\end{thm}
\begin{remark}
From Theorem \ref{theorem1}, we get the asymptotics of $u$ when $\alpha_0\in [1,\infty[$, that is as $r\to 1^-$, $u(r)\sim \frac{\lambda}{\alpha_0-1}(1-r)^{-\alpha_0+1}$ if $\alpha_0> 1$ and $u(r)\sim -\lambda\ln(1-r)$ if $\alpha_0=1$. $\alpha_0<1$ holds when $\delta<-mp-(p-1-\beta)$ and then one falls in the case of (iii) of Theorem \ref{khaled4} whereas assertion (iv) is verified  as $\alpha_0\geq 1$.
\end{remark}
%\begin{remark}
%The new restrictions imposed on the parameters guarantee that the associated asymptotically autonomous system $ \eqref{eq14} $ is $ C^1 $.
%\end{remark}
\noindent To prove theorem \ref{theorem1}, we perform several changes of variables that reduce  \eqref{otoo35} to an asymptotically autonomous dynamical system for which we appeal the rich theory of cooperative systems (see \cite{book30}, \cite{book23}, \cite{book31} for further details on the subject). We also use chain-reccurence properties of asymptotically-autonomous flows (see \cite{book32}) to identify the $\omega$-limit set of the non autonomous dynamical system. We also need properties concerning equation \eqref{otoo35} or \eqref{otoo36} as comparison principle and continuity with respect the initial data. Section 2 is dedicated to prove these tools whereas Section 3  deals the proof of Theorem \ref{theorem1}. We also give at the end of Section 3 some numerical simulations associated to asymptotics proved in Theorem \ref{theorem1}.
\section{Technical results}
More precisely, we study in this section the pairs of positive functions $(U,v)$ which satisfy, for some $R>0$
\begin{equation}\label{uniq}
\left\{\begin{array}{ll}
( U^{p-1-\alpha})' +\dfrac{\gamma}{r}U^{p-1-\alpha}= \dfrac{\gamma}{N-1}v^m & \text{in} \quad (0,R), \\
\left( (v')^{p-1} \right)' +\dfrac{N-1}{r}(v')^{p-1}=v^\beta U^{q} & \text{in} \quad (0,R),\\
%v(0)=v_0 \mbox{ and } 
v'(0)=U(0)=0 \mbox{ and } v\to +\infty  \mbox{ as } r\to R. &
\end{array}\right.
\end{equation}
We start with a comparison principle satisfied by \eqref{uniq}.
\begin{pro}\label{lemn}
\indent Let $ (U,v) $ and $ (\widehat{U},\widehat{v}) $ be positive
functions %that blow up respectively at $ R $ and $ \widehat{R} $
which satisfy
\begin{equation}\label{otooe1}
\left\{
\begin{array}{ll}
\left( r^{\gamma}U^{p-1-\alpha} \right)'\geq \dfrac{\gamma}{N-1}r^{\gamma}v^m & \text{for any} \quad r\in (0,R), \\
\left( r^{N-1}(v')^{p-1} \right)' \geq r^{N-1}v^\beta U^{q} & \text{for any} \quad r\in (0,R),\\
v(0)=v_0 \mbox{ and } v'(0)=U(0)=0 & \\
U>0,\ v'>0 \mbox{ in } (0,R)&
%\vspace{0.2cm}
%u(0)=a, v(0)=b .
\end{array}
\right.
\end{equation}
and
\begin{equation}\label{otooe2}
\left\{
\begin{array}{ll}
\left( r^{\gamma}\widehat U^{p-1-\alpha} \right)'\leq \dfrac{\gamma}{N-1}r^{\gamma}\widehat  v^m & \text{for any} \quad r\in (0,\widehat R), \\
\left( r^{N-1}(\widehat v')^{p-1} \right)' \leq r^{N-1}\widehat v^\beta\widehat U^{q} & \text{for any} \quad r\in (0,\widehat R),\\
\widehat v(0)=\widehat v_0 \mbox{ and } \widehat v'(0)=\widehat U(0)=0 & \\
\widehat U>0,\ \widehat v'>0 \mbox{ in } (0,\widehat R)&
%\vspace{0.2cm}
%u(0)=a, v(0)=b .
\end{array}
\right.
\end{equation}
for some $R$, $\widehat R>0$.\\
If $ \widehat{v}(0)<v(0) $, then $\widehat{v}<v$ and $\widehat U<U$ in $ (0,\min(R, \widehat R)) $.\\
In addition, assuming that %$v$  blows up at $R<\infty$  and 
$\widehat v$ blows up at $\widehat R$, we have %then $\min(R,\widehat R)=R$ and more precisely
$R<\widehat R$.
\end{pro}
\begin{proof}
Define $\rho^*=\sup \{\rho>0 \ |\  \widehat{v}(r)<v(r)\mbox{ on } (0,\rho)\}$. Since $ \widehat{v}(0)<v(0) $, $\rho^*>0$ hence from the first equations of \eqref{otooe1}-\eqref{otooe2}, we deduce $U>\widehat U$ on $(0,\rho^*)$. Moreover using  the second equations of \eqref{otooe1}-\eqref{otooe2}, we get $v'>\widehat v'$ on $(0,\rho^*)$. Thus the mapping $v-\widehat v$ is increasing $(0,\rho^*)$ and we deduce $\rho^*=\min (R,\widehat R)$ and $\widehat v<v$ on $(0,\rho^*)$.\\
Now, assume that $\widehat v$ blows up at $\widehat R$. Since $\widehat v<v$ on $(0,\min (R,\widehat R))$, we have $R\leq \widehat R$.\\
Consider the pair $(U_\lambda, {v}_{\lambda})$ where $U_\lambda=\lambda^{-\alpha_0}U(\frac{.}{\lambda})$ and $v_\lambda$ is defined in \eqref{scal}. Thus $ \left( {U}_{\lambda},{v}_{\lambda} \right) $ satisfies \eqref{otooe1} in $(0,\lambda R) $.\\
Choosing $\lambda$ such that $\dfrac{\widehat v(0)}{v(0)}\leq \lambda ^{-\beta_0}<1$ then $ {v}_{\lambda}(0) >\widehat v(0) $ and from the first part, we deduce $v_\lambda>\widehat v$ on $(0,\min(\lambda R,\widehat R)$ and $  \lambda R\leq \widehat R$ . 
\end{proof}
\begin{pro}\label{lem2} The mapping $\Phi:v_0\to R$, where $(0,R)$ is the maximal interval of a solution $(U,v)$ of \eqref{uniq} with $v(0)=v_0$, is continuous on $(0,+\infty)$.
\end{pro}
\begin{proof} 
Let $(v_{0,n})\subset \R$ such that $v_{0,n}\to v_0$ and . We define $(U,v)$ and $R$ (respectively $(U_n,v_n)$ and $R_n$) solutions of \eqref{uniq} with $v(0)=v_{0}$ defined on the maximal interval $(0,R)$ (respectively with $v_n(0)=v_{0,n}$ defined on the maximal interval $(0,R_n)$).\\  
Let $\varepsilon>0$, Then, we define 
$$ \lambda_\pm=\left(1\pm\dfrac{\varepsilon}{v_0}\right)^{-\frac{1}{\beta_0}}$$
and  the pair $(U_{\lambda_{\pm}}, v_{\lambda_{\pm}})=\left(\lambda^{-\alpha_0}_{\pm}U(\frac{.}{\lambda_{\pm}}),\lambda^{-\beta_0}_{\pm}v(\frac{.}{\lambda_{\pm}})\right) $  is solution of \eqref{uniq} on $(0,\lambda_\pm R)$ with $v_{\lambda_{\pm}}(0)=v_0\pm \epsilon$.
For $n$ large enough, $v_{0,n}\in (v_0-\varepsilon,v_0+\varepsilon)$, hence from Lemma \ref{lemn}, we have $\lambda_+R< R_n < \lambda_-R$ and we deduce $R_n\to R$ as $n\to +\infty$.
\end{proof}
\begin{cor}
For any $R>0$, there exist an unique $v_0>0$ and an unique pair of positive functions $(U,v)$ solution to \eqref{uniq} such that $v(0)=v_0$.
\end{cor}
\begin{proof}
%Noting that the system \eqref{uniq} comes from \eqref{otoo36} with $U=u'$,
Lemma \ref{lemn}-\ref{lem2} imply that the mapping $\Phi$ is decreasing and continuous on $(0,+\infty)$. Thus $\Phi$ is invertible on $(0,+\infty)$ and we deduce the uniqueness.
\end{proof}
\begin{remark}\label{form}
Using the scaling property and the previous results, we have, for any $R>0$, $\Phi^{-1}(R)=\Phi^{-1}(1)R^{-\beta_0}$ and the unique solution $(U_R,v_R)$ on $(0,R)$ of \eqref{uniq} is given by 
\begin{equation*}
U_R=R^{-\alpha_0}  U_1\left( \dfrac{.}{R}\right) \quad \mbox{and} \quad v_R= R^{-\beta_0} v_1 \left( \dfrac{.}{R}\right)
\end{equation*} 
where $(U_1,v_1)$ is the unique solution of \eqref{uniq} on $(0,1)$.
\end{remark}
\begin{remark}\label{stabi}
Let $\varepsilon>0$ and let  $(U,v)$ and $(U_\varepsilon,v_\varepsilon)$ two pairs of solutions of \eqref{uniq} such that $ v(0)-\tilde v(0)=\pm\varepsilon$. From Remark \ref{form}, there exists $\lambda_\varepsilon$ such that $\lambda_\varepsilon\to 1 $ as $\varepsilon \to 0$, 
$\tilde v_\varepsilon=\lambda_\varepsilon^{-\beta_0}v(\frac{.}{\lambda_\varepsilon})$ and $\tilde U_\varepsilon=\lambda_\varepsilon^{-\alpha_0}U(\frac{.}{\lambda_\varepsilon})$.\\
By the proof of Lemma \ref{lemn}, the sequences $(v_\varepsilon)_\varepsilon$, $(U_\varepsilon)_\varepsilon$ and $(v'_\varepsilon)_\varepsilon$ are monotone and converge everywhere to $v$, $U$ and $v'$ on the compact sets of $[0,R)$. By Dini's Theorem, we deduce the uniform convergence on the compact sets of $[0,R)$.
\end{remark}

\section{ Asymptotic behavior and proof of Theorem $ \ref{theorem1} $}
Through this section, we only consider $\delta<0$, $\alpha_0>0$ and $R=1$. Moreover we have $\alpha_{0}(p-1-\alpha)+1$, $\beta_0$ and $\gamma_0$ are positive. At the end of the section, we  give the proof of Theorem \ref{theorem1}.\\
Let  $ (u,v) $  be a pair of nonconstant positive radial solutions of \eqref{otoo23} given by Theorem \ref{khaled4}. Then, setting $ U=u'$ and $ V=v'$ in  \eqref{otoo36}, $ U $ and $V$ are positive, increasing on $(0,1)$ and satisfy 
\begin{equation}\label{eq1}
\begin{cases}
\left( U^{p-1-\alpha}\right)' +\dfrac{\gamma}{r}U^{p-1-\alpha}=\frac{\gamma}{N-1} v^m
& \text{for } \hspace{0.2 cm}r\in (0,1),\\
\left( V^{p-1}\right)' +\dfrac{N-1}{r} V^{p-1}=v^\beta U^{q}
& \text{for } \hspace{0.2 cm}r\in (0,1),\\
U(0)=V(0)=0. &
\end{cases}
\end{equation}
Now, set
\begin{center}
$ U(r)=\dfrac{\lambda }{(1-r)^{\alpha_{0}}}a(r) $, $ v(r)=\dfrac{\mu }{(1-r)^{\beta_{0}}}b(r) $ and $ V(r)=\dfrac{\nu }{(1-r)^{\gamma_{0}}}c(r) $
\end{center}
where 
$ \lambda$, $\mu$, $\nu>0$ to be fixed later. Thus, $ \eqref{eq1} $ leads to
\begin{equation}\label{eq2}
\begin{cases}
\left( \left( \frac{\lambda a(r)}{(1-r)^{\alpha_{0}}} \right)^{p-1-\alpha}\right)' +\frac{\gamma}{r}\left( \frac{\lambda a(r)}{(1-r)^{\alpha_{0}}} \right)^{p-1-\alpha}=\frac{p-1-\alpha}{p-1}\left( \frac{\mu b(r)}{(1-r)^{\beta_{0}}} \right)^m
& \text{for } \hspace{0.2 cm}r\in (0,1), \\
\left( \left( \frac{\nu c(r)}{(1-r)^{\gamma_{0}}} \right)^{p-1}\right)' +\frac{N-1}{r} \left( \frac{\nu c(r)}{(1-r)^{\gamma_{0}}} \right)^{p-1}=\left( \frac{\mu b(r)}{(1-r)^{\beta_{0}}} \right)^\beta \left( \frac{\lambda a(r)}{(1-r)^{\alpha_{0}}} \right)^{q}
& \text{for } \hspace{0.2 cm}r\in (0,1).
\end{cases}
\end{equation}
Simple computations give $\alpha_{0}(p-1-\alpha)+1=m\beta_{0}$ and 
$$
\left( \left( \dfrac{\lambda a(r)}{(1-r)^{\alpha_{0}}} \right)^{p-1-\alpha}\right)' =  (p-1-\alpha)\lambda^{p-1-\alpha} \dfrac{(1-r)a'(r)+\alpha_{0} a(r)}{(1-r)^{m\beta_0}}a(r)^{p-2-\alpha} .$$
Thus, multiplying the first equation of \eqref{eq2} by $ (1-r)^{m\beta_0 }$ 
choosing $\lambda$ and $\mu$ such that $ \lambda^{p-1-\alpha}\alpha_{0}(p-1-\alpha)=\frac{\gamma}{N-1}\mu^{m}$, this yields for any $r\in (0,1)$ 
\begin{equation}\label{eq5}
(1-r)a^{p-\alpha-2}((p-1-\alpha)a' +\dfrac{\gamma}{r}a) =\alpha_{0}(p-1-\alpha)\left( b^m -a^{p-1-\alpha}\right).
\end{equation}
In the same way, noting that $\gamma_{0}(p-1)+1=\beta \beta_{0}+ \alpha_{0}q$ and  multiplying the second equation of \eqref{eq2} by \break $ (1-r)^{\beta\beta_0+\alpha_0q }$, we get  
\begin{equation}\label{eq8}
(1-r)c^{p-2}((p-1)c' +\dfrac{N-1}{r} c)=\gamma_{0}(p-1) \left( b^{\beta}a^{q}-c^{p-1} \right)
\end{equation}
where $\lambda$, $\mu$ and $\nu$ satisfy $ \mu^{\beta}\lambda^{q}= \gamma_{0}(p-1) \nu^{p-1} $.\\ 
Moreover, since $V= v'$, we deduce that
\begin{equation}\label{eq9}
\mu (b'(r)(1-r)+\beta_{0} b(r))=\nu c(r).
\end{equation}
Choosing $\mu$ and $\nu$ such that $\beta_{0} \mu=\nu $, $ \eqref{eq9} $ becomes
\begin{equation}\label{eq10}
 b'(r)(1-r)=\beta_{0}\left( c(r)-b(r) \right).
\end{equation}
Hence, we deduce that the triplet $( a,b,c)$ satisfies
\begin{equation}\label{eq11}
\begin{cases}
(1-r)a^{p-\alpha-2}((p-1-\alpha)a' +\dfrac{\gamma}{r}a) =\alpha_{0}(p-1-\alpha)\left( b^m -a^{p-1-\alpha}\right),\\
\vspace{0.2cm}
(1-r)c^{p-2}((p-1)c' +\dfrac{N-1}{r} c)=\gamma_{0}(p-1) \left( b^{\beta}a^{q}-c^{p-1} \right),\\
\vspace{0.2cm}
(1-r)b'=\beta_{0}\left( c-b \right)
\end{cases}
\end{equation}
under the following relations
\begin{equation*}%\label{eq12}
%\begin{cases}
\lambda^{p-1-\alpha}\alpha_{0}=\dfrac{\mu^{m}}{p-1},\qquad
%\vspace{0.2cm}
\mu^{\beta}\lambda^{q}= \gamma_{0}(p-1) \nu^{p-1},\qquad
%\vspace{0.2cm}
\beta_{0} \mu=\nu.
%\end{cases}
\end{equation*}
which leads to
\begin{equation}\label{eq13}
\begin{cases}
\lambda=\left( \gamma_0(p-1)\beta_0^{p-1}(\alpha_0(p-1))^{\frac{p-1-\beta}{m}}\right)^{-\frac{m}{\delta}},\\
\vspace{0.2cm}
\mu= \left(\gamma_0(p-1)\beta_0^{p-1}(\alpha_0(p-1))^{\frac{q}{p-1-\alpha}}\right)^{-\frac{p-1-\alpha}{\delta}}, \\
\vspace{0.2cm}
\nu=\beta_{0} \left(\gamma_0(p-1)\beta_0^{p-1}(\alpha_0(p-1))^{\frac{q}{p-1-\alpha}}\right)^{-\frac{p-1-\alpha}{\delta}}.
\end{cases}
\end{equation}
%Now, we focus on the first equation and the third equation of $ \eqref{eq11} $ that we transform as an asymptotically-autonomous dynamical system in $ \mathbb{R}^3 $.\\
Consider now the following change of variable $r=1-e^{-t}$ {\it i.e.} $t=- \ln (1-r)$ and we define $\mathcal X $, $\mathcal  Y $ and $\mathcal  Z $ as follows
\begin{center}
$\mathcal  X(t)=a^{p-1-\alpha}(r) $, $\mathcal  Y(t)=b(r) $ and $\mathcal  Z(t)=c^{p-1}(r). $
\end{center}
Hence, $( \mathcal X,\mathcal  Y,\mathcal  Z) $ are positive functions satisfying the following $ C^1$-asymptotically autonomous system on $  [t_0,+\infty)$ for some $t_0>0$:
\begin{equation}\label{eq14}
\left\{\begin{array}{l}
X'(t)+\dfrac{\gamma}{e^{t}-1}X(t)=\alpha_{0}(p-1-\alpha)(|Y|^{m-1}(t)Y(t)-X(t)) %&\mbox{ on } [t_0,\infty)
\\ 
Y'(t)=\beta_{0}(|Z|^{\frac{1}{p-1}-1}(t)Z(t)-Y(t)) %&\mbox{ on } [t_0,\infty)
\\
Z'(t)+\dfrac{N-1}{e^{t}-1}Z(t)=\gamma_{0}(p-1)(|Y|^{\beta-1}(t)Y(t)|X|^{\frac{q}{p-1-\alpha}-1}(t)X(t)-Z(t)) %&\mbox{ on } [t_0,\infty)
\\
\end{array}\right.
\end{equation}
for some $(X(t_0),Y(t_0),Z(t_0))\in \R^3$. \noindent The asymptotic dynamical system associated to  \eqref{eq14} is given by: 
\begin{equation}\label{2eq14}
\left\{\begin{array}{l}
X'(t)=\alpha_{0}(p-1-\alpha)(| Y |^{m-1}Y-X\\
Y'(t)=\beta_{0}(| Z |^{\frac{1}{p-1}-1}Z-Y) \\
Z'(t)=\gamma_{0}(p-1)(| Y |^{\beta-1}Y | X |^{\frac{q}{p-1-\alpha}-1}X-Z) 
\end{array}\right.
\end{equation} 
The system $\eqref{2eq14} $ can be rewritten as $\zeta'(t)=g(\zeta)$, $\zeta(t_0)=(X(t_0),Y(t_0),Z(t_0)\in \R^3$ where
\begin{equation}\label{eq16}
\zeta(t)=\left(
\begin{array}{c}
X(t)\\
Y(t)\\
Z(t)
\end{array}\right) \ \text{and}\ g(\zeta)=\left(
\begin{array}{c}
\alpha_{0}(p-1-\alpha)(| Y |^{m-1}Y-X)\\
\beta_{0}(| Z |^{\frac{1}{p-1}-1}Z-Y)\\
\gamma_{0}(p-1)(| Y |^{\beta-1}Y | X |^{\frac{q}{p-1-\alpha}-1}X-Z) 
\end{array}\right) .
\end{equation}
The system \eqref{2eq14} is cooperative, $ \mathrm{div}\,g<0 $ and we have the following properties.
\begin{pro}\label{nature of equilibrium} The following assertions hold:
\begin{enumerate}
\item $\left( 0, 0, 0 \right)$ and $\left( 1, 1, 1\right)$ are the only equilibrium points of the system $ \eqref{2eq14} $. 
\item $\left( 0, 0, 0 \right)$ is asymptotically stable {(and then a sink)}.
\item $\left( 1, 1, 1\right)$ is {an hyperbolic saddle point}.
\end{enumerate}
\end{pro}
\begin{proof} Let $ P=\left(X_e,Y_e,Z_e \right) $ be an equilibrium point of the system \eqref{2eq14}, then $g(P)=0$ which implies that 
\begin{equation}\label{zizo1}
\left\{
\begin{array}{l}
|Y_e|^{m-1}Y_e=X_e,\\
|Z_e|^{\frac{1}{p-1}-1}Z_e=Y_e,\\
|Y_e|^{\beta-1}|X_e|^{\frac{q}{p-1-\alpha}-1}Y_eX_e=Z_e.
\end{array}\right.
\end{equation}
Proof of {\it 1.}: Obviously,  $ Y_e=0 $ yields $ X_e=Z_e=0 $ and solve the system \eqref{zizo1}.\\
If $ Y_e \neq 0$, from \eqref{zizo1}, we deduce that $ X_e,Y_e $ and $ Z_e $ have the same sign and that\break  $ | X_e|^{\frac{q}{p-1-\alpha}-1}X_e=| Y_e|^{p-2-\beta}Y_e $. Thus, the first equation of $ \eqref{zizo1} $ implies that $ Y_e^{\frac{\delta}{q}}=1 $. Hence, we get $ Y_e=X_e=Z_e=1 $.\\
\ \\
%we deduce that $\left( 0, 0, 0 \right)$ and $\left( 1, 1, 1\right)$ are the only equilibrium points of the system \eqref{2eq14}.\\
Proof of {\it 2.}: The assertion follows directly looking the linearized matrix at $\left( 0, 0, 0 \right)$ given by: 
\begin{center}
$M_0=\left( \begin{array}{ccc}
-\alpha_{0}(p-1-\alpha) & 0 & 0\\
\vspace{0.2cm}
0 & -\beta_{0} & 0\\
\vspace{0.2cm}
0 & 0 & -\gamma_{0}(p-1)
\end{array} \right)$.
\end{center}
Proof of {\it 3.}:
The linearized matrix at $\left( 1, 1, 1 \right)$ is given by
\begin{center}
$M_1= \left( \begin{array}{ccc}
-\alpha_{0}(p-1-\alpha) & m\alpha_{0}(p-1-\alpha) & 0\\
\vspace{0.2cm}
0 & -\beta_{0} & \dfrac{\beta_{0}}{p-1}\\
\vspace{0.2cm}
\dfrac{q \gamma_{0}(p-1)}{p-1-\alpha} & \beta \gamma_{0}(p-1) & -\gamma_{0}(p-1)
\end{array} \right) $.
\end{center}
Thus, $det( \lambda I - M_1)  = \lambda^{3}+C_{1}\lambda^{2}+C_{2}\lambda+ C_{3} $ where
\begin{center}
$\begin{array}{l}
C_{1}=\alpha_{0}(p-1-\alpha)+\beta_{0}+\gamma_{0}(p-1),\\
C_{2}=\beta_{0} \gamma_{0}(p-1-\beta)+\alpha_{0}(p-1-\alpha)(\beta_{0}+\gamma_{0}(p-1)),\\
C_{3}=\alpha_{0}\beta_{0} \gamma_{0}\delta.
\end{array}$
\end{center}
Since $ \beta_{0}+1=\gamma_{0} $, $ \alpha_{0}(p-1-\alpha)+1= \beta_{0}m $ and $ \gamma_{0}(p-1)+1=\beta \beta_{0}+\alpha_{0}q$, we get $ C_{1}+C_{2}+C_{3}=-1$.
Hence $ \lambda_{1}=1 $ is an eigenvalue of $ M_1 $.\\
Let us call $ \lambda_{2} $, $ \lambda_{3} $ the two other eigenvalues of $ M_1 $. Then we have $\lambda_{2}+\lambda_{3}=-C_{1}-1<0$ and $  \lambda_{2}\lambda_{3}=-C_{3}=-\alpha_{0}\beta_{0} \gamma_{0}\delta >0 $.\\
Thus, $ Re(\lambda_{2}),\ Re(\lambda_{3})<0 $ and we deduce that $\left( 1, 1, 1 \right)$ is hyperbolic point with $2$-dimensional stable manifold.
\end{proof}
\noindent We have the following properties about the dynamical system \eqref{eq14}:
\begin{pro}\label{pos}
Let $T>0$ and $t_0\in (0,T)$. Assume $(X,Y,Z)$ a triplet of functions satisfying \eqref{eq14} on $[t_0,T)$ with $(X(t_0),Y(t_0),Z(t_0))\in \IT\R^3_+$. Then, for any $t\in [t_0,T)$, $(X(t),Y(t),Z(t))\in \IT \R^3_+$.
\end{pro}
\begin{proof}
Define $T^*=\sup \{t>t_0 \ |\ X(s)>0,\ Y(s)>0 \mbox{ and } Z(s)>0 \mbox{ for any } s\in (0,t)\}$ and assume that $T^*<T$.\\
From the second equation of \eqref{eq14}, we deduce that for any $t\in [t_0,T^*)$, $(e^{\beta_0 t}Y(t))'>0$ and hence $e^{\beta_0t}Y(t)>e^{\beta_0t_0}Y(t_0)>0$. This implies that $Y(T^*)>0$.\\
Assuming $X(T^*)=0$, the first equation of \eqref{eq14} yields for $t=T^*$, $X'(T^*)=\alpha_{0}(p-1-\alpha)Y^{m}(T^*)>0$. This implies there exists a small neighbourhood $(T^*-\eta,T^*+\eta)$ of $T^*$ such that $X'$ is positive on $(T^*-\eta,T^*+\eta)$. Since $X(T^*)=0$, we deduce $X(t)<0$ on $(T^*-\eta, T^*)$ which contradicts the definition of $T^*$.\\
In the same, using the third equation of \eqref{eq14}, we can not have $Z(T^*)=0$ and hence we conclude $T^*=T$.
\end{proof}
\begin{remark}\label{marchencore}
Proposition \ref{pos} holds considering a triplet of functions satisfying the autonomous system \eqref{2eq14}.
\end{remark}
\begin{pro}\label{inf1}
Let $T>0$ and $t_0\in (0,T)$. Let $(X,Y,Z)$ be a triplet of functions satisfying \eqref{eq14} on $[t_0,T)$ with $(X(t_0), Y(t_0),Z(t_0))\in \IT \R^3_+$. Assume %the exists $t_0\in (0, T)$ such that 
$X(t_0)<1$, $Y(t_0)<1$ and $Z(t_0)<1$. Then, for any $t\in [t_0,T)$, $ X(t)<1$, $Y(t)<1 $ and $ Z(t)<1 $.
\end{pro}
\begin{proof}
By Proposition \ref{pos},  for any $t\in [t_0,T)$, $(X(t),Y(t),Z(t))\in \IT \R^3_+$ and hence we deduce
\begin{equation*}
    \left\{
\begin{array}{l}
(X-1)'(t) <  \alpha_0 (p-1-\alpha)\left( Y^m-1-(X-1)\right),\\
(Y-1)'(t)=  \beta_0 \left( Z^{\frac{1}{p-1}}-1-(Y-1)\right),\\
(Z-1)'(t) <  \gamma_0 (p-1)\left( Y^\beta X^{\frac{q}{p-1-\alpha}}-1-(Z-1)\right)
  \end{array}\right.
\end{equation*}
and we get
\begin{equation}\label{name}
    \left\{
\begin{array}{l}e^{-\alpha_0 (p-1-\alpha)t}\left( e^{\alpha_0 (p-1-\alpha)t}(X-1)\right)' <\alpha_0 (p-1-\alpha)( Y^m-1),\\
e^{-\beta_0 t}\left(e^{\beta_0 t}(Y-1)\right)' = \beta_0 \left( Z^{\frac{1}{p-1}}-1\right),\\
e^{-\gamma_0 (p-1)t}\left(e^{\gamma_0 (p-1)t}(Z-1)\right)' <  \gamma_0 (p-1)\left( Y^\beta X^{\frac{q}{p-1-\alpha}}-1\right).
  \end{array}\right.
\end{equation}
Define $\tau=\sup\{t>t_0\ |\ X(t)<1,\ Y(t)<1 \mbox{ and } Z(t)<1 \mbox{ on } (t_0,t)\}$.\\
Assume $\tau<T$ then $X(\tau)=1$ or $Y(\tau)=1$ or $Z(\tau)=1$. \\
In any cases, the third equation of \eqref{name} at $t=\tau$ implies that $Z(t)<1$ on the interval $ [\tau,\tau+\eta] $ for some $ \eta > 0 $. Hence, if $Z(\tau)=1$, we get a contradiction.\\
Otherwise, using the previous estimate in the second equation of \eqref{name}, we get $Y(\tau)<1$ which would contradict $Y(\tau)=1$. Finally, using $Y(\tau)<1$ in the first equation of \eqref{name}, this yields $X(\tau)<1$ and a contradiction with $X(\tau)=1$.
\end{proof}
\noindent For $t_0$ small enough, $(\mathcal X(t_0),\mathcal Y(t_0),\mathcal Z(t_0))\in \IT \R^3_+$, hence the triplet $(\mathcal X,\mathcal Y,\mathcal Z)$ satisfies the $C^1$-asymptotically autonomous system
\begin{equation}\label{eq14b}
\left\{\begin{array}{l l}
X'(t)+\dfrac{\gamma}{e^{t}-1}X(t)=\alpha_{0}(p-1-\alpha)(Y^{m}(t)-X(t)) &\mbox{ on } [t_0,\infty)\\ 
Y'(t)=\beta_{0}(Z^{\frac{1}{p-1}}(t)Z(t)-Y(t)) &\mbox{ on } [t_0,\infty)\\
Z'(t)+\dfrac{N-1}{e^{t}-1}Z(t)=\gamma_{0}(p-1)(Y^{\beta}(t)X^{\frac{q}{p-1-\alpha}}(t)-Z(t)) &\mbox{ on } [t_0,\infty)\\
\end{array}\right.
\end{equation}
with $(X(t_0),Y(t_0),Z(t_0))\in \IT\R^3_+$. Hence we get the following asymptotic dynamical system:
\begin{equation}\label{2eq14b}
\left\{\begin{array}{l}
X'=\alpha_{0}(p-1-\alpha)(Y^{m}-X) \\
Y'=\beta_{0}(Z^{\frac{1}{p-1}}-Y) \\
Z'=\gamma_{0}(p-1)(Y^{\beta}X^{\frac{q}{p-1-\alpha}}-Z).
\end{array}\right.
\end{equation} 
Now, we study the trajectories $\left(\mathcal X,\mathcal Y,\mathcal  Z \right)$. The first result is the relatively compactness:
\begin{pro}\label{prop1}
The solution $\left(\mathcal X,\mathcal Y,\mathcal  Z \right)$ is bounded as $ t \to +\infty $.
\end{pro}
\begin{proof}
Firstly, assume that $\mathcal  Y $ is bounded, then we have from the first equation of  \eqref{eq14b} \\
\begin{center}
$\mathcal  X'(t)+\alpha_{0}(p-1-\alpha)\mathcal X \leq  C,$
\end{center}
%that is
%\begin{center}
%$  \left( e^{\alpha_{0}(p-1-\alpha)t} X(t) \right)' \leq  Ce^{\alpha_{0}(p-1-\alpha)t}$
%\end{center}
which is equivalent to
\begin{center}
$ \left( e^{\alpha_{0}(p-1-\alpha)t}\mathcal  X \right)' \leq  C  e^{\alpha_{0}(p-1-\alpha)t}.$
\end{center}
Integrating this last inequality from $ t_0 $ to $ t $ leads to
\begin{center}
$  e^{\alpha_{0}(p-1-\alpha)t}\mathcal  X(t)-e^{\alpha_{0}(p-1-\alpha)t_0}\mathcal  X(t_0)  \leq  \dfrac{C}{\alpha_{0}(p-1-\alpha)} \left( e^{\alpha_{0}(p-1-\alpha)t}-e^{\alpha_{0}(p-1-\alpha)t_0} \right). $
\end{center}
Since $\mathcal X$ is positive, this implies that $\mathcal  X $ is bounded. A similar argument implies that $\mathcal  Z $ is bounded.\\
It remains to prove that $\mathcal  Y $ is actually bounded and for this, we proceed by contradiction.\\
Consider  $ \left( \widehat{u},\widehat{v} \right)$ the solution of \eqref{otoo36} with 
$\widehat{u}(0)=u(0)$ and $\widehat{v}(0)> v(0)$ defined in $ (0,\widehat{R})$. From Proposition  \ref{lemn}, it follows that $ \widehat R< 1 $. We define $ ( \widehat{X}, \widehat{Y}, \widehat{Z} ) $ on $(0,-\ln(1-\widehat R))$ the transform of $(\widehat u,\widehat v)$ constructed as previously.\\
We also define $\tilde{X}(t)=\mathcal X(t+T_0) $, $  \tilde{Y}(t)=\mathcal Y(t+T_0) $ and $  \tilde{Z}(t)=\mathcal Z(t+T_0) $
with $ T_0 $ large enough such that $\mathcal  Y( T_0) > \widehat{Y}(0) $ since $\mathcal Y$ is unbounded.
Then $(\tilde u,\tilde v)$ the reverse transform of $ (\tilde{X}, \tilde{Y}, \tilde{Z}) $ satisfies
\begin{equation*}
\begin{cases}
\left( (\tilde u')^{p-1-\alpha}\right)' +\dfrac{\gamma}{r+\tilde r_0}(\tilde u')^{p-1-\alpha}=\dfrac{\gamma}{N-1}\tilde v^m
& \text{in} \quad (0,1), \\
\vspace{0.2cm}
\left( (\tilde v')^{p-1}\right)' +\dfrac{N-1}{r+\tilde r_0}(\tilde v')^{p-1}=v^\beta(\tilde u')^{q} & \text{in} \quad (0,1),\\
\tilde u'(0)\geq 0,\ \tilde v'(0)\geq0,\ \tilde u,\,\tilde v>0& \text{in} \quad (0,1),
\end{cases}
\end{equation*}
where $\tilde r_0=\dfrac{r_0}{1-r_0}$ with $r_0=1-e^{-T_0}$.\\
%For any $r\in (0,1)$, $\tilde u(r)=(1-r_0)^{1-\alpha_0}u((1-r_0)(r+\tilde r_0))$ and  $\tilde v(r)=(1-r_0)^{-\beta_0}v((1-r_0)(r+\tilde r_0))$.\\
Noting $(\tilde u',\tilde v)$ satisfies \eqref{otooe1} and $\tilde v(0)> \widehat v(0)$, we apply Proposition \ref{lemn} with $(\widehat u',\widehat v)$, we deduce $\widehat v<\tilde v$ in\break  $(0,min(1,\widehat R))=(0,\widehat R)$. Since $\widehat v$ blows up at $\widehat R$, we get the contradiction.
\end{proof}
%{\noindent We have the following property.
\noindent From proposition \ref{prop1}, we deduce the convergence of the trajectories $(\mathcal X,\mathcal Y,\mathcal Z)$:
\begin{pro}\label{prop2} We have
$$\lim_{t\to +\infty} (\mathcal X(t),\mathcal Y(t),\mathcal Z(t))=(1,1, 1).$$
\end{pro}
\noindent For the proof of Proposition \ref{prop2}, we need to recall two notions of dynamical systems:
\begin{define}
A circuit is a finite sequence of equilibria $z_1, z_2, . . . , z_K = z_1$, ($K \geq  2$) such that $W^u(z_i)\cap W^s(z_{i+1})\neq \emptyset$ where $W^u, W^s$ denote the stable and unstable manifolds of each equilibria.
\end{define}

\begin{define}
Let $ X \subset \mathbb{R}^3 $ be a nonempty positively invariant subset for an autonomous semiflow $ \phi $
and $ x, y \in X $.
\begin{enumerate}
\item[(i)]  For $\epsilon > 0$ and $t > 0$, an $(\epsilon, t)$-chain from $ x \in X $ to $ y \in X $ is a sequence of points in $ X $,
$ x = x_1, x_2, . . . , x_n, x_{n+1} = y$ and of times $t_1, t_2, . . . , t_n \geq  t$ such that $| \phi(t_i
, x_i) - x_{i+1}| < \epsilon$.
\item[(ii)] A point $ x \in X $ is called chain recurrent if for every $\epsilon > 0$, $t > 0$ there is
an $(\epsilon, t)$-chain from $x$ to $x$ in $X$.
\item[(iii)] The set $X$ is said to be chain recurrent if every point $ x \in X $ is chain
recurrent in $ X $.
\end{enumerate}
\end{define}
\begin{proof}
Setting  $\zeta_{0}=\left(\mathcal X(t_0),\mathcal Y(t_0), \mathcal Z(t_0) \right)\in \IT \R^3_+$, we define  the $ \omega $-limit set $\omega_0=\omega (t_0,\zeta_{0})\subset \R^3_+$ of $\phi$ the semiflow of the asymptotically autonomous system \eqref{eq14b}.\\
From Proposition \ref{prop1}, $\omega_0$ is bounded. Thus Theorem $1.8$ in \cite{book32} implies $\omega_0$ is nonempty, compact and connected. Moreover $\omega_0$ is invariant for the semiflow denoted by $ \phi_A $ associated to the asymptotic autonomous system \eqref{2eq14b} , {\it i.e.} for any $t\geq 0$
$$ \phi_A(t,\omega_0)= \omega_0$$
and $\omega_0$ is chain recurrent for $\phi_A$.\\
Finally, $\omega_0$ satisfies $dist( \phi(t,t_0,\zeta_{0}),\omega_0) \to 0$ as $ t \to +\infty $.\\
{\bf Step 1:} $\omega_0 \subset \lbrace (0,0,0);(1,1,1)\rbrace$.\\
For that we argue by contradiction: assume that there exists $P_0 \in \omega_0\backslash  \{(0,0,0),(1,1,1)\}$ .\\
Then, $\phi_A(t,P_0)\in \omega_0$ for any $t\geq 0$ thus the trajectory $t\to\phi_A(t,P_0)$ is bounded as $t\to +\infty$. From Proposition 1.2 in \cite{book32}, we have $\omega_{\phi_A}(P_0)$ is invariant. Hence, we deduce for any $t<0$:
\begin{equation*}
\begin{split}
\omega_{\phi_A}(P_0)&=\phi_A(0,\omega_{\phi_A}(P_0))=\phi_A(t,\phi_A(-t,\omega_{\phi_A}(P_0))\\
&=\phi_A(t,\omega_{\phi_A}(P_0)).
\end{split}
\end{equation*}
Thus we deduce that $t\to\phi_A(t,P_0)$ is bounded as $t\to -\infty$ and the limit sets of $P_0$ associated to $\phi_A$ is bounded.\\
Then, applying  Theorem 10  in \cite{book31} together with $ \mathrm{div}\,g <0 $ and Proposition  \ref{nature of equilibrium}, we infer that
\begin{equation}\label{*}
\lim_{|t|\to+\infty} \phi_A(t,P_0)\in  \{(0,0,0),(1,1,1)\}.
\end{equation}
Now, since $ (0,0,0) $ is asymptotically stable, this implies $\phi_A(t,P_0) \to (1,1,1)$ as $ t \to -\infty$ and we have two following cases:\\
\textbf{Case 1:} heteroclinic orbit {\it i.e} $\phi_A(t,P_0) \to (0,0,0)$ as $t \to + \infty$ \\
or\\
\textbf{Case 2:} homoclinic orbit {\it i.e} $ \phi_A(t,P_0) \to (1,1,1) $ as $ t \to +\infty $.\\
Since $\omega_0$ is invariant by the flow $\phi_A$, both cases imply that $ \omega (\zeta_{0}) $ contains either a connecting orbit between the two different equilibria ($P_0\in W_u(1,1,1$)) or a cycle with respect to $\phi_A$.
%of an hyperbolic equilibrium ($(1,1,1)$) with itself (homoclinic orbit)
\\
In the first case,  the heteroclinic orbit included in $\omega_0$ is not chain recurrent since $(0,0,0)$ is asymptotically stable which contradicts $\omega_0$ is chain recurrent. %(reference? c'est d\'emontr\'e dans le rapport de Bachir mais il n'y a pas de r\'ef\'erence).}\textcolor{magenta}{c'est \'evident qu'une trajectoire reliant deux points d'\'equilibre distincts dont un est asymptotiquement stable n'est pas chain reccurrent (tu ne pourras pas revenir pr\`es du premier point d'\'equilibre) sinon tu peux citer \cite{book32}.}\\
\\
For the second case, $(1,1,1)$ is hyperbolic with transverse unstable and stable manifolds, there can not be any circuit (see for instance \cite{book30} page 1228 or \cite{book31} pages 1677-1678) and we have a contradiction.\\
\ \\
{\bf Step 2:} $ \omega_0=(1,1,1) $\\
Since $ \omega_0 $ is connected then either $ \omega_0= \lbrace (0,0,0) \rbrace  $ or $ \omega_0 = \lbrace (1,1,1) \rbrace  $.\\
The first possibility does not hold. Indeed arguing by contradiction, then there exists $T_0>0$ such that  $\mathcal X(T_0)<1$, $\mathcal Y(T_0)<1$ and $\mathcal Z(T_0) < 1 $.\\
{
Consider $(\tilde u, \tilde v)$ a couple of blow-up solutions of \eqref{otoo36} with the initial data $\tilde v(0)=v(0)+\varepsilon$ on $(0, R_\epsilon)$ where $R_\epsilon<1$ by Proposition \ref{lemn}. Moreover from Proposition \ref{lem2}, we have $ R_\varepsilon$ goes to $1$ as $\varepsilon \to 0$.\\
%We define then $(\tilde X,\tilde Y,\tilde Z)$ using the same transformations as previously $ %\tilde{Y} $ blows at some $ \tilde{T}= -\ln (1-\tilde{R})$.\\
Let $R\in (0,1)$, we define $\mathcal T_1,\ \mathcal T_2$ and $\mathcal T_3$ from $C([0,R])$ to $C([0,-\ln(1-R)])$ three linear and continuous operators as follows:  for any $w\in C([0,R])$, for any $t\in [0,T]=[0,-\ln(1-R)]$, 
$$\mathcal T_1(w)(t)=\frac{e^{-\alpha_0t}}{\lambda}w(1-e^{-t}),\ \mathcal T_2(w)(t)=\frac{e^{-\beta_0t}}{\mu}w(1-e^{-t})\ \mbox{and} \ \mathcal T_3(w)(t)=\frac{e^{-\gamma_0t}}{\nu}w(1-e^{-t}).$$
By Remark \ref{stabi}, we get as $\varepsilon \to 0$:
$$\sup_{[0,T]}|\mathcal X-\tilde X|=\sup_{[0,T]}|\mathcal T_1(u')-\mathcal T_1(\tilde u')|\leq C\sup_{[0,R]}|u-\tilde u'|\to 0,$$ 
$$\sup_{[0,T]}|\mathcal Y-\tilde Y|=\sup_{[0,T]}|\mathcal T_2(v)-\mathcal T_2(\tilde v)|\leq C\sup_{[0,R]}| v-\tilde v|\to 0$$
and
$$\sup_{[0,T]}|\mathcal Z-\tilde Z|=\sup_{[0,T]}|\mathcal T_3(v')-\mathcal T_3(\tilde v')|\leq C\sup_{[0,R]}|v'-\tilde v'|\to 0.$$
We deduce, for $\varepsilon$ small enough that then there exists  $R\in (1-e^{-T_0}, R_\varepsilon)$,
$$\sup_{[0,T]}|\mathcal X-\tilde X|< 1-\mathcal X(T_0),\ \sup_{[0,T]}|\mathcal Y-\tilde Y|< 1-\mathcal Y(T_0)\ \mbox{and}\ \sup_{[0,T]}|\mathcal Z-\tilde Z|< 1-\mathcal Z(T_0). $$
Hence $\tilde X(T_0)<1$, $\tilde Y(T_0)<1$ and $\tilde Z(T_0)<1$.\\
From Proposition \ref{inf1}, we deduce that, for any $t\in (T_0,-\ln(1-R_\varepsilon))$,  $\tilde X(t)<1$, $\tilde Y(t)<1$ and $\tilde Z(t)<1$.\\
Finally, by construction of $\tilde Y$, we have $\tilde Y(t)\to +\infty$ as $t\to -\ln(1-R_\varepsilon)$ and we obtain a contradiction.
}\end{proof}
\begin{figure}[h!]
\hspace{-1cm}
		\includegraphics[width=0.6\textwidth]{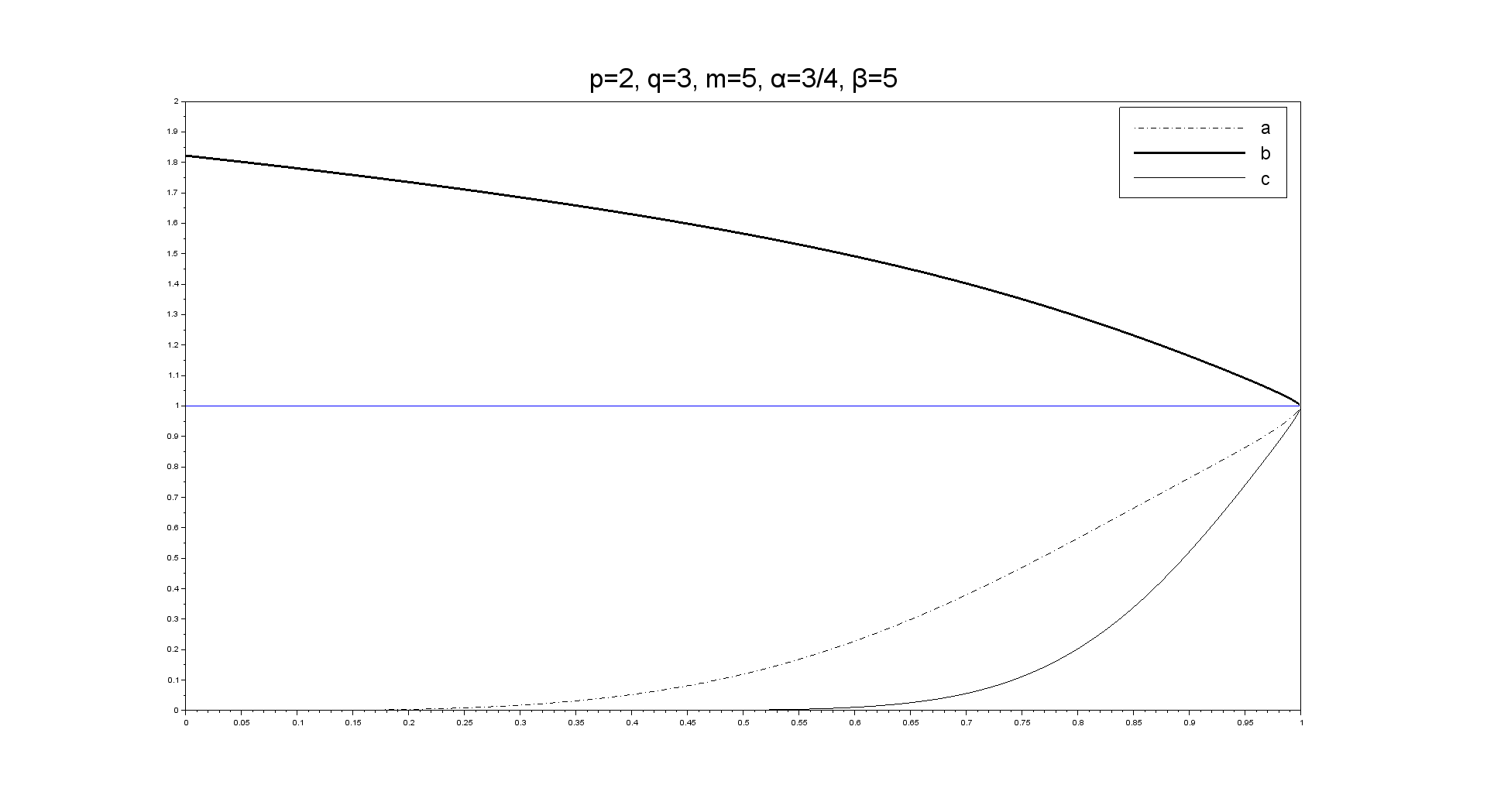}\hspace{-1cm}
				\includegraphics[width=0.6\textwidth]{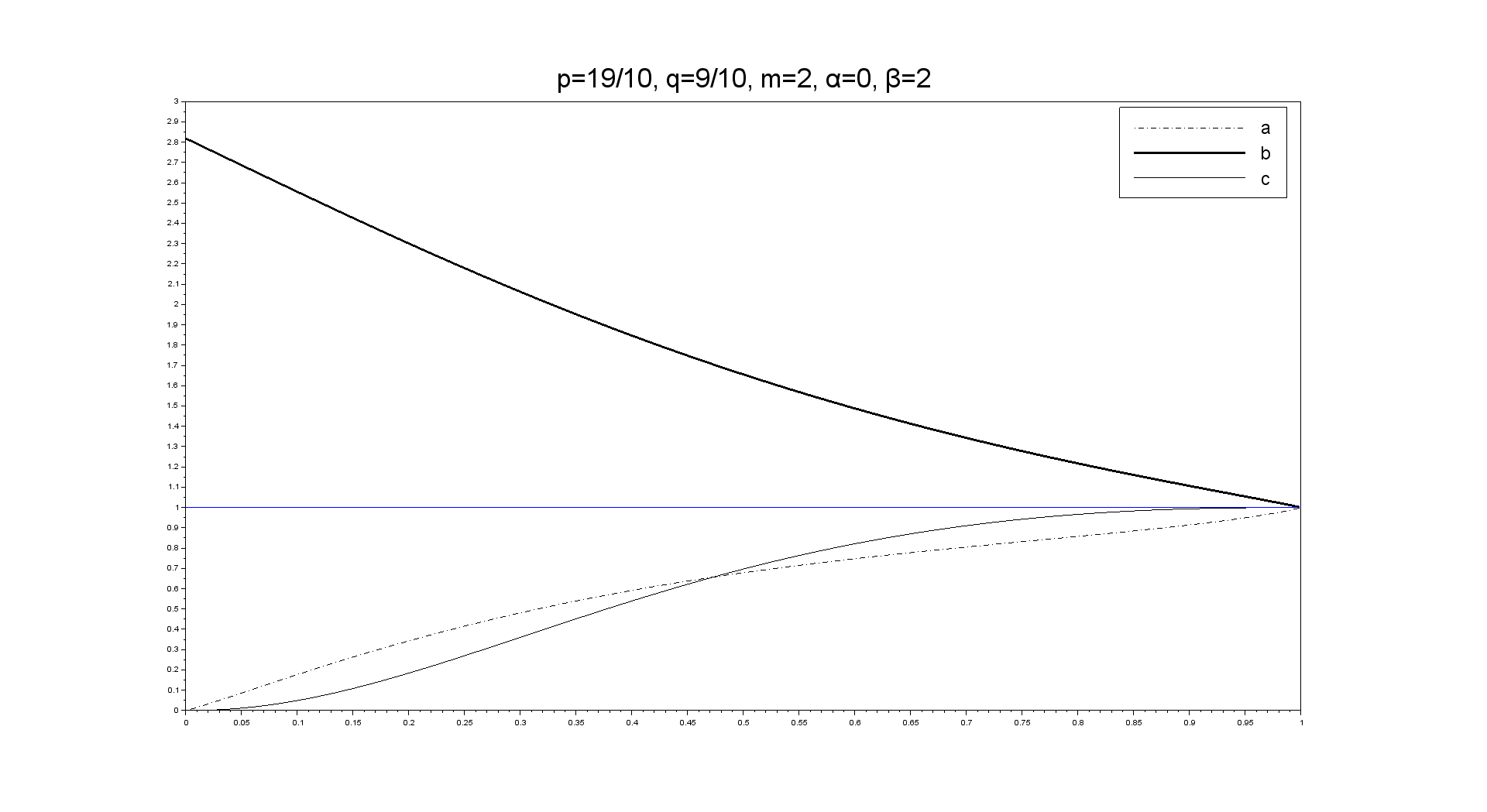}
				\caption{$\alpha_0<1$}
				\label{fig1}
\end{figure}
\begin{figure}[h!]
\hspace{-1cm}
				\includegraphics[width=0.6\textwidth]{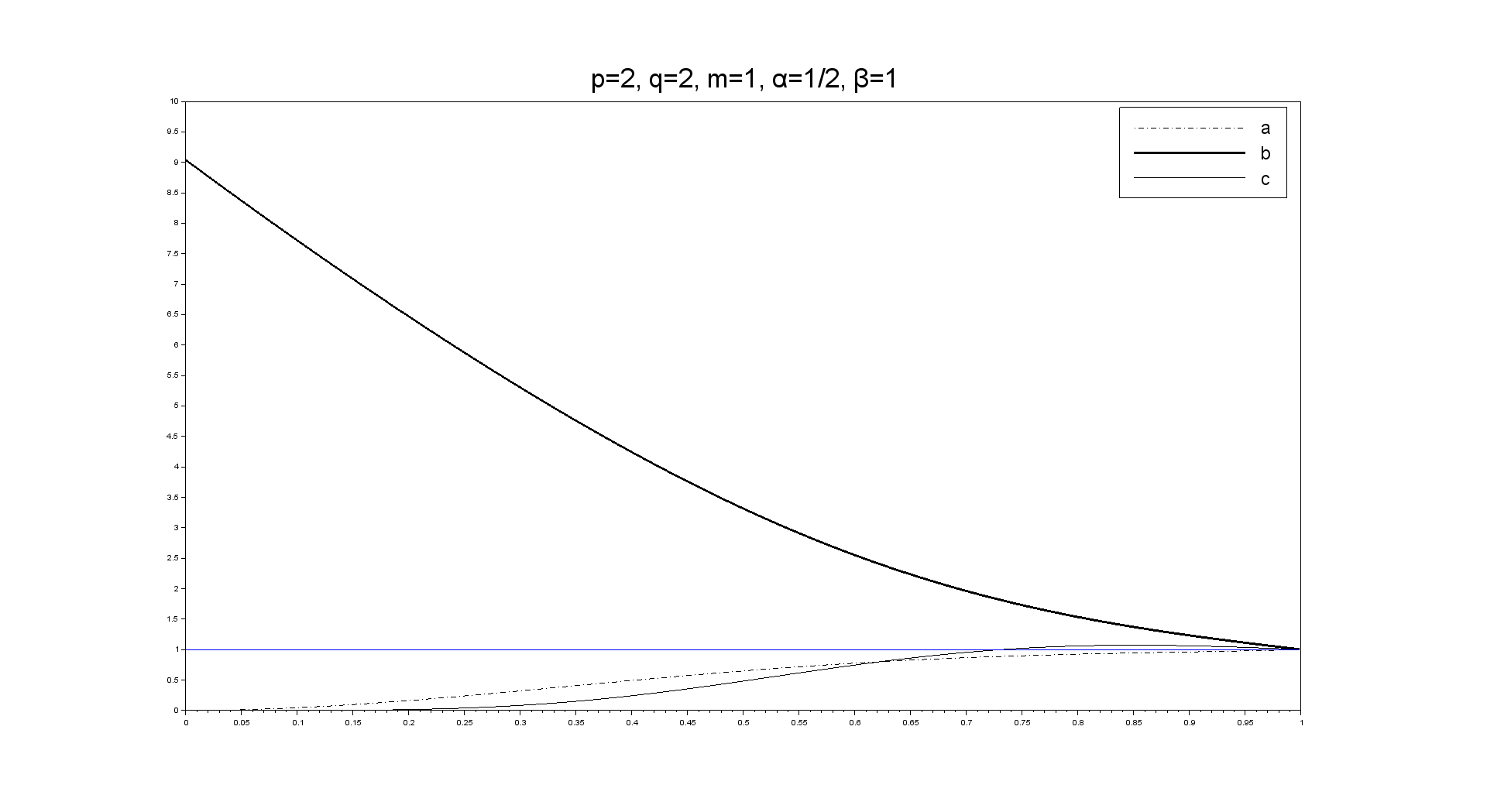}\hspace{-1cm}
				\includegraphics[width=0.6\textwidth]{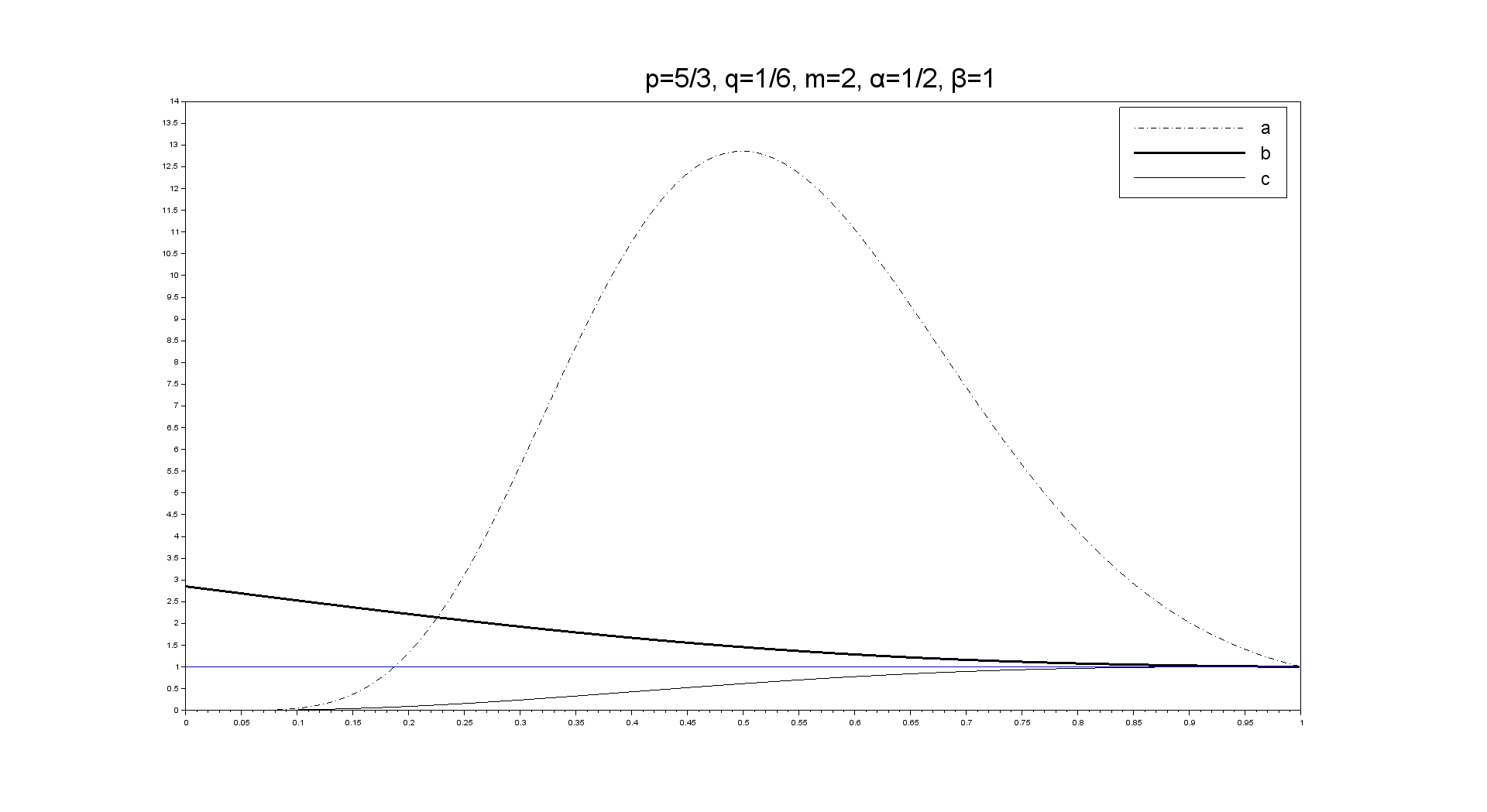}
		  	\caption{$\alpha_0\geq 1$}
				\label{fig2}
\end{figure}
\begin{figure}[h!]
\hspace{-1cm}
		\includegraphics[width=0.6\textwidth]{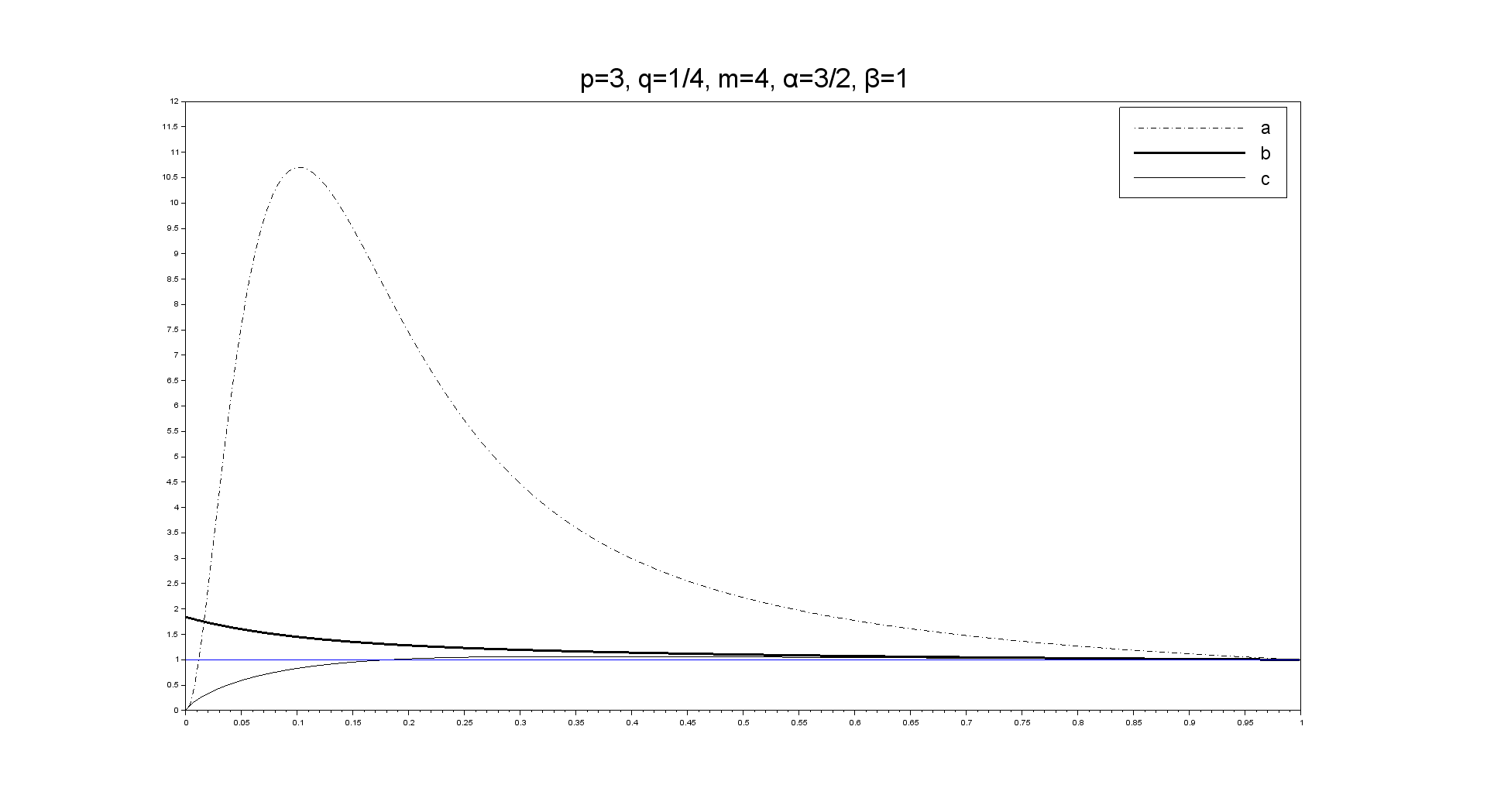}\hspace{-1cm}
		\includegraphics[width=0.6\textwidth]{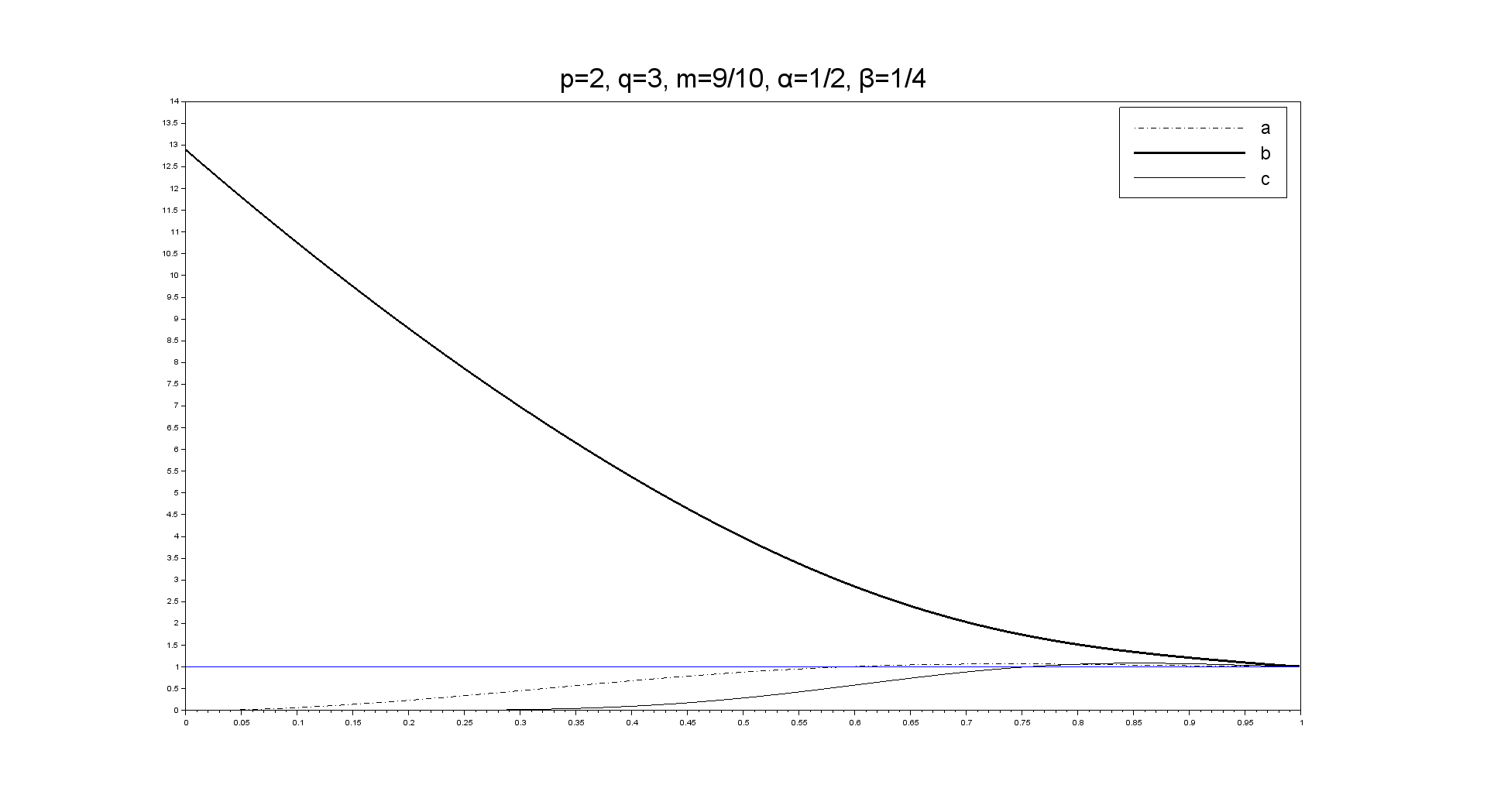}
							\caption{Other examples}
				\label{fig3}
\end{figure}
\begin{remark}
Figures \ref{fig1}-\ref{fig3} give examples of convergence of the triplet functions $(a,b,c)$  to $(1,1,1)$ as $r\to 1$ where $(a,b,c)$ satisfies \eqref{eq11}. More precisely, we have $\alpha_0<1$ in Figure \ref{fig1} and  $\alpha_0\geq1$ in Figure \ref{fig2}. We give particular examples for the conditions of Theorem \ref{theorem1} in Figure \ref{fig3}: in the left-hand graph, we choose $p> 2$ and $q<p-1-\alpha$ and in the right-hand graph, $\beta<1$ and $ m<1$. For every graph, $\delta<0$, $b_0>0$, $\alpha_0>0$ and $N=3$.
\end{remark}

\noindent\textbf{Proof of Theorem $\ref{theorem1}$:}\\
%Now, we  prove the asymptotics of any blow up solution.\\
%Thanks to Theorem \ref{khaled4}, we have
%\begin{enumerate}
%\item[(i)] No positive solution with $ u(1)= \infty $, $ v(1)< \infty $.
%\item[(ii)] $ u(1)< \infty $ and $ v(1)< \infty $ if and only if $\delta>0$.
%\item[(iii)] $ u(1)< \infty $ and $ v(1)= \infty $ if and only if $\delta<-mp$.
%\item[(iv)] $ u(1)=v(1)= \infty $ if and only if $\delta \in [-mp,0)$.
%\end{enumerate}
We get from Proposition \ref{prop2} that, as $r \to 1^-$
$$ a(r)=\dfrac{U(r)}{\lambda}(1-r)^{\alpha_0} \to 1,\quad b(r)=\dfrac{v(r)}{\mu}(1-r)^{\beta_0} \to 1\  \mbox{ and }\ c(r)=\dfrac{V(r)}{\nu}(1-r)^{\gamma_0} \to 1$$
hence 
$$u'(r) \sim \dfrac{\lambda}{(1-r)^{\alpha_0}}\ \mbox{ and }\ 
 v(r) \sim \dfrac{\mu}{(1-r)^{\beta_0}}  \hspace{0.1cm} \text{as} \hspace{0.1cm} r \to 1^- .$$
Since $ \beta_0>0 $, we have obviously $v(1)=\infty$. Finally we deduce 
%we have
%\begin{center}
%$\begin{array}{ccc}
%\alpha_0 < 1 & \Leftrightarrow & 1- \alpha_0>0\\
%& \Leftrightarrow & \dfrac{\delta+p(m+1)-(1+\beta)}{\delta}>0\\
%& \Leftrightarrow &\dfrac{-qm +(p-1-\alpha)(p-1-\beta)+pm+p-(1+\beta)}{\delta}\\
%&=&\dfrac{-qm +pm+(p-\alpha)(p-1-\beta)}{\delta}>0.
%\end{array}$\\
%\end{center}
%Since $ \delta <0 $, we find that 
%\begin{equation}\label{cond}
%\alpha_0 < 1  \Leftrightarrow qm > pm+(p-\alpha)(p-1-\beta).
%\end{equation}
%Similarly, we have the following equivalence
%\begin{equation}\label{cond2}
%\delta\neq 0\mbox{ and }\alpha_0 \geq  1  \Leftrightarrow (p-1-\alpha)(p-1-\beta) < qm < pm+(p-\alpha)(p-1-\beta).
%\end{equation}
$$ u(r) \sim \dfrac{\lambda}{(\alpha_0-1)(1-r)^{\alpha_0-1}}\quad \mbox{ as } r \to 1^-  $$
if $\alpha_0\neq 1$ and
$$ u(r) \sim \lambda \ln (1-r)\quad \mbox{ as } r \to 1^-  $$
if $\alpha_0= 1$.
%where
%\begin{center}
%$\begin{cases} \alpha_0=\dfrac{-p(m+1)+(1+\beta)}{\delta} \\ \beta_0=-\dfrac{p(p-1-\alpha)+q}{\delta} \end{cases}$ and $\begin{cases} A=\left( (p-1)^{\frac{q+p-1-\alpha}{p-1-\alpha}+\frac{\delta}{m}}\alpha_{0}^{\frac{q}{p-1-\alpha}+\frac{\delta}{m}}\gamma_{0}\beta_{0}^{p-1}\right)^{-\frac{m}{\delta}} \\  B= \left( (p-1)^{\frac{q+p-1-\alpha}{p-1-\alpha}}\alpha_{0}^{\frac{q}{p-1-\alpha}}\gamma_{0}\beta_{0}^{p-1}\right)^{-\frac{p-1-\alpha}{\delta}}. \end{cases}$
%\end{center}

\renewcommand{\bibname}{REFERENCES}

\end{document}